\documentclass{amsart}

\usepackage{latexsym}
\usepackage{amsfonts}
\usepackage{amsmath}
\usepackage{amssymb,amsxtra,amscd}
\usepackage{mathrsfs}
\usepackage{enumerate}

\newcommand{\R}{\mathbb{R}}

\newcommand{\N}{\mathbb{N}}

\newcommand{\K}{\mathbb{K}}

\newcommand{\CC}{\mathscr{C}}

\newcommand{\SSa}{\mathscr{S}}

\newtheorem{theorem}{Theorem}

\newtheorem{corollary}[theorem]{Corollary}

\newtheorem{example}[theorem]{Example}

\title{Strongly continuous semigroups on some Fr\'echet spaces}

\author[L.\ Frerick]{Leonhard Frerick}
\address{Fachbereich IV -- Mathematik, Universit\"{a}t Trier, D-54286 Trier, Germany}
\email{frerick@uni-trier.de}

\author[E.\ Jord\'a]{Enrique Jord\'a}

\thanks{The research was partially done during a stay of the fourth named author at EPSA-UPV. This visit was supported by Proyecto Prometeo II/2013/013.
The research of the first and second named authors was supported by MICINN and FEDER, Project MTM2010–15200.
The research of the second named author was partially supported by Programa de Apoyo a la Investigación y Desarrollo de la UPV PAID-06-12. }

\address{Departamento de Matem\'atica Aplicada, E. Polit\'{e}cnica Superior
de Alcoy, Universidad Polit\'ecnica de Valencia, Plaza Ferr\'andiz
y Carbonell 2, E-03801 Alcoy (Alicante), Spain}
\email{ejorda@mat.upv.es}

\author[Th.\ Kalmes]{Thomas Kalmes}
\address{Fachbereich IV -- Mathematik, Universit\"{a}t Trier, D-54286 Trier, Germany}
\email{kalmesth@uni-trier.de}

\author[J.\ Wengenroth]{Jochen Wengenroth}
\address{Fachbereich IV -- Mathematik, Universit\"{a}t Trier, D-54286 Trier, Germany}
\email{wengenroth@uni-trier.de}

\begin{document}

\begin{abstract}
We prove that for a strongly continuous semigroup $T$ on the Fr\'echet space $\omega$ of all scalar sequences, its generator is
a continuous linear operator $A\in L(\omega)$ and that, for all $x\in\omega$ and $t\geq 0$,
the series $\exp(tA)(x)=\sum\limits_{k =0}^\infty\frac{t^k }{k !}A^k  (x)$ converges to $T_t(x)$.
This solves a problem posed by Conejero. Moreover, we improve recent results of Albanese, Bonet, and Ricker about semigroups on
strict projective limits of Banach spaces.
\end{abstract}

\maketitle

\section{Introduction}

In \cite{Conejero} Conejero asked whether on $\omega=\K^\N$ every strongly
continuous semigroup $T$ is of the form $T_t(x)=\sum\limits_{k =0}^\infty\frac{t^k }{k  !}A^k (x)$
for a continuous linear operator $A$ on $\omega$. This question arose in the
context of hypercyclicity: It is shown in \cite[Theorem 2.7]{Conejero} that no such semigroup
on $\omega$ can be hypercyclic. Although Shkarin \cite{Shkarin} proved 
that there are not even supercyclic  strongly continuous semigroups on $\omega$,
the question of Conejero remained open. Only a partial answer is contained in \cite{AlBoRi}.

\medskip

The definitions and most basic results for semigroups on locally convex spaces $X$ are
the same as for Banach spaces, we refer to \cite{Komatsu,Komura,Ouchi,Yosida}. A strongly continuous semigroup
$T$ on $X$ is thus a morphism from the semigroup $([0,\infty),+)$ to that of continuous linear
operators $(L(X),\circ)$ such that all orbits $t\mapsto T_t(x)$ are continuous.
If the convergence $T_t \to id_X$ for $t\to 0$ is uniform on bounded subsets of $X$, the semigroup is called uniformly continuous.

 $T$ is called
locally equicontinuous if
$\{T(s): 0\leq s\leq t\}$ is equicontinuous for every $t>0$, i.e., for every continuous seminorm $p$ on $X$ there is another continuous seminorm $q$ on $X$
such that $p(T_s(x)) \leq q(x)$ for all $x\in X$ and $0\leq s\leq t$. On barrelled spaces,
every strongly continuous semigroup is already locally equicontinuous because of the Banach-Steinhaus theorem.
The (infinitesimal) generator $(A,D(A))$ of a strongly continuous semigroup on a locally convex
space is defined as in the Banach space setting as the derivative of the orbit at $0$.

\medskip

Although under rather weak assumptions (sequential completeness to have a vector valued integral and barrelledness to apply the uniform boundedness principle)
many results from the Banach space setting carry over to strongly continuous semigroups on locally convex spaces there are some crucial differences because
the exponential series $\exp(A)(x)=\sum\limits_{k=0}^\infty A^k( x)/k!$ need not converge for continuous linear operators. Therefore, one does not always have 
the familiar representation $T_t=\exp(tA)$ for semigroups with continuous generators:

\begin{example} \rm
Consider $\CC^\infty(\R)$ with its usual topology and the strongly continuous semigroup defined by $T_t(f)(x)=f(x+t)$.
Then $D(A)=\CC^\infty(\R)$ and $A(f)=f'$. For any $f$ which is flat at the origin but does not vanish on $(0,\infty)$ the
series $\exp(tA)(f)=\sum\limits_{k =0}^\infty\frac{t^k }{k  !} f^{(k )}$ cannot converge to $T_t(f)$ because for
$t>0$ with $f(t)\neq 0$ we have
\[
f(t)=T_{t}(f)(0)\neq \sum_{k =0}^\infty\frac{t^k }{k  !} f^{(k )}(0)=0.
\] 
Using E.\ Borel's theorem, that a smooth function can have any given sequence of derivatives at $0$, 
one also gets $f\in \CC^\infty(\R)$ such that the
exponential series diverges.
\end{example}

Answering Conjero's question, we will prove that such phenomena do not occur on the Fr\'echet space $\omega$.

\section{Semigroups on strict projective limits}

Let $X_n$ be Banach spaces with norms $\|\cdot\|_n$, $\pi_m^n:X_m \to X_n$ norm decreasing operators for $n\le m\in\N$ with
$\pi_m^n\circ \pi_k^m=\pi_k^n$ as well as $\pi_n^n=id_{X_n}$, and 
\[
X=\{(x_n)_{n\in\N}\in\prod_{n\in\N} X_n: \pi_m^n(x_m)=x_n \text{ for all } n\le m\}
\]
 its projective limit.
Every Fr\'echet space has such a representation, and $X$ is called a {\em quojection} if there is a representation
with surjective (hence open) $\pi_m^n$. Countable products of Banach spaces are of this form, in particular
spaces like $L^p_{loc}(\Omega)$ of locally $p$-integrable functions or $\CC^n(\Omega)$ for open sets $\Omega\subseteq \R^d$ and $n\in\N_0$, but there are 
quojections which are not isomorphic to a product.

If $X$ has a strict representation as above then, by a simple induction, $\pi_m:X\to X_n$, $(x_n)_{n\in\N}\mapsto x_m$ are also surjective.
Applying this observation to the spaces $\ell^\infty_I(X_n)$ of bounded functions $I\to X_n$ one obtains that
$\pi_m$ lifts bounded sets, that is, there are bounded sets $D_n\subseteq X$ such that $\pi_n(D_n)$ contains the unit ball $B_n$
of the Banach space $X_n$. This lifting property was first proved in \cite{DiZa}.

The following theorem improves results of Albanese, Bonet, and Ricker  \cite{AlBoRi}  who showed its first part under restrictive additional assumptions.
Moreover, their proposition 3.2 shows that the strong assumption on the space $X$ in our theorem cannot be relaxed too much: On every nuclear K\"othe space 
(and even every locally convex 
space containing a complemented copy of such a space) there are equicontinuous, uniformly continuous semigroups whose generator is not everywhere
defined and discontinuous.

\pagebreak[4]

\begin{theorem}\label{main result} Let $X$ be a quojection.
\begin{enumerate}[(1)]
\item  If $T$ is a uniformly continuous semigroup on $X$ then its generator is continuous and everywhere
	defined, and for all $x\in X$ and $t\geq 0$ we have
	\[T_t(x)=\exp(tA)(x)=\sum_{k =0}^\infty\frac{t^k }{k  !}A^k ( x).\] 

\item $A\in L(X)$ generates a strongly continuous semigroup if and only if
	\[\forall\,n\in\N\;\exists\, m\in\N\;\forall\,k \in\N_0:\; \pi_m(x)=0 \;\Longrightarrow\; \pi_n(A^k (x))=0.\]
Then the generated semigroup is even uniformly continuous.
\end{enumerate}
\end{theorem}

\begin{proof}
   As $X$ is barrelled the semigroup is locally equicontinuous, so that, for every $t_0>0$ and $n\in\N$, there are $m\ge n$ and $c>0$ such that 
for all $x\in X$ and $t\in [0,t_0]$
\[
\|\pi_n(T_t(x))\|_n \le c \|\pi_m(x)\|_m.
\]
In particular, we have
\[
(\ast)\quad  \pi_m(x) =0 \;\Longrightarrow\; \pi_n(T_t(x))=0 \text{ for all } t\le t_0.
\]
As in the case of semigroups on Banach spaces it is easily seen that the Cesaro means 
\[ C_t(x)=\frac 1t \int_0^t T_s(y)ds\]
belong to $D(A)$, $A(C_t(y)) = \frac 1t (T_t(y)-y)$, and $C_t(y)\to y$ holds uniformly on bounded sets since $T$ is uniformly continuous. 
Moreover, $C_t$ satisfy the same continuity estimates as $T_t$.
Therefore, the following operators are well-defined   and continuous for $t\in (0,t_0]$:
\[
\tilde{C}_t: X_m \to X_n \; \pi_m(y) \mapsto \pi_n(C_t(y)).
\]
Since $\pi_m$ lifts bounded sets, every $z$ in the unit ball $B_m$ of $X_m$ can be represented as $z=\pi_m(y)$ with $y\in D_m$ and,
as $C_t \to id$ uniformly on $D_m$, we obtain that $\tilde C_t$ converges uniformly on $B_m$ to $\pi_m^n$.
Since the set of surjective operators is open in $L(X_m,X_n)$, see e.g.\
\cite[chapter XV, theorem 3.4]{Lang}, we conclude that $\tilde C_t$ is surjective for some sufficiently small $t>0$.

If now $ n_0\in \N$ is given we take $n\ge n_0$ with $\pi_n(z)=0$ $\Rightarrow$ $\pi_{n_0}(T_t(z))=0$ and then
 $m\ge n$ again with ($\ast$) for $t_0=1$, say.  Given $x\in X$ we choose $y\in X$ with $\pi_n(x)= \tilde C_t(\pi_m(y)) =\pi_n(C_t(y))$
so that $\pi_{n_0}(T_h(x))=\pi_{n_0}(T_h(C_t(y)))$ for small $h$.
Therefore, 
\[
\pi_{n_0}\left(\frac 1h (T_h(x)-x)\right) = \pi_{n_0}\left( \frac 1h ( T_h(C_t(y))-C_t(y))\right)
\]
 converges to $\pi_{n_0}(  \frac 1t (T_t(y)-y))$.

This shows that the difference quotients satisfy the Cauchy condition so that the completeness of $X$ implies that $A(x)$ is
defined for every $x\in X$. Moreover, $A$ is continuous either because of the closed graph theorem or because of 
$\pi_{n_0}(A(x))=  \pi_{n_0}(  \frac 1t (T_t(y)-y))$ together with the fact that $y$ can be chosen with $\|\pi_m(y)\|\le\tilde  c \|\pi_n(x)\|_n$
using that $\tilde C_t$ is open.

Since $A^k (x)$ is the $ k $th derivative of the orbit $t\mapsto T_t(x)$ at $0$ we obtain from ($\ast$) that $A$ satisfies the
  condition in the second part of the theorem (this argument uses only local equicontinuity and is thus true for strongly continuous
semigroups).

\medskip
We will now show that under the condition of (2) the exponential series \[\exp(tA)(x)=\sum\limits_{k =0}^\infty \frac{t^k }{k !}A^k (x)\]
converges absolutely and uniformly on bounded sets. Since then $\exp(tA)$ is a uniformly continuous semigroup with generator $A$ we have proved both parts of the theorem because 
a locally equicontinuous semigroup is uniquely determined by its generator.

We define linear maps $\tilde A_0=\pi_m^n$ and 
\[ \tilde A_k: X_m \to X_n, \; \pi_m(x) \mapsto \pi_n(A^k(x)) \]
so that $\tilde A_k \circ \pi_m =\pi_n \circ A^k$ and $\tilde A_{k+1}\circ \pi_m = \tilde A_k\circ \pi_m \circ A$.

As above we take a bounded set $D_m \subseteq X$ with $B_m\subseteq \pi_m(D_m)$. The continuity of $A$ then implies
$\pi_m \circ A(D_m) \subseteq \lambda B_m$ for some $\lambda>0$. We now  claim that
\[ 
\| \tilde A_k\|_{L(X_m,X_n)} \le \lambda^k
\]
for all $k\in \N_0$. For $k=0$ this holds because $\tilde A_0=\pi_m^n$ is norm decreasing. If the claim is true for some $k\in \N_0$
and $y\in B_m$ is given we choose $x\in D_m$ with $\pi_m(x)=y$ and obtain
\[
\|\tilde A_{k+1}(y)\|_n = \|\tilde A_{k+1}\circ \pi_m(x)\|_n = \|\tilde A_k (\pi_m(A(x)))\|_n \le \lambda^k \|\pi_m(A(x))\|_m \le \lambda^{k+1}.
\]
Finally, we obtain $\|\pi_n(A^k(x))\|_n = \|\tilde A_k(\pi_m(x))\|_n \le \lambda^k \|\pi_m(x)\|_m$ which clearly implies the
absolute convergence of the exponential series.
\end{proof}

The very strong form of the convergence of the exponential series shown at the end of the proof need not hold in  arbitrary Fr\'echet spaces.
The shift $T_t(f)(z)=f(z+t)$ on the space of entire functions is generated by $A(f)=f'$ and is of the form $\exp(tA)$ because of
the Taylor representation of entire functions. However, one cannot estimate $\| A^k(f)\|_n \le \lambda^k \|f\|_m$.

\medskip

Let us remark that part of the implication in (2), namely that a  strongly continuous semigroup whose generator is everywhere defined is uniformly continuous,
holds for all barrelled spaces. This follows easily from the Banach-Steinhaus theorem for $\{(T_h - id)/h: h\in[0,1]\}$. Moreover, the equicontinuity of this family characterizes semigroups with a continuous generator.

\medskip

The countability of the projective spectrum was only used to obtain barrelledness and that $\pi_m:X\to X_m$ lifts bounded sets. The theorem is
thus true for non-countable strict projective limits of Banach spaces which are barrelled and satisfy this lifting property. In particular,
it holds for arbitrary products
of Banach spaces. Moreover, as in \cite[proposition 3.4]{AlBoRi}  it is easy to extend the results to so-called prequojections $X$ where only the bidual $X''$
has a strict representation.

\medskip
Since strongly continuous,  locally equicontinuous semigroups on Montel spaces are easily seen to be uniformly continuous we 
immediately obtain the answer to Conejero's question:

\begin{corollary} Every strongly continuous semigroup on an arbitrary product $\K^I$ has a continuous generator $A$ and is of the form
$\exp(tA)$.
\end{corollary}

\begin{example}
\begin{rm}
	Consider the  backward shift $A:\omega \to \omega$, $x\mapsto (x_2,x_3,\ldots)$. Theorem \ref{main result} implies that 
           $A$ does not generate a strongly continuous semigroup on $\omega$. On the other hand,
	the operator $B(x)=(x_2,0,x_4,0,x_6,0,\ldots)$ satisfies $B^2=0$ as well as $(A-B)^2=0$ so that
	both operators do generate strongly continuous semigroups on $\omega$, given by $T(t)=id +tB$
	and $S(t)=id+t(A-B)$, respectively. It is thus not only the ``shape" of the matrix determined by the operator
which decides whether it generates a semigroup.
\end{rm}
\end{example}

Let us finally remark that the second part of theorem \ref{main result} implies that, 
if a continuous operator $A$ on a quojection generates a uniformly continuous semigroup, then the same is true for $A^2$.
This does not hold for the Fr\'echet space $\CC^\infty(\R)$: $A(f)=f'$ generates the shift semigroup but $B(f)=A^2(f)=f''$ does not
generate a strongly continuous semigroup on $\CC^\infty(\R)$ : 

The Gau{\ss}-Weierstra{\ss} semigroup $\tilde T$ on $\SSa(\R)$ given by the convolution with the Gau{\ss} 
kernel $k(t,x)=(\pi t)^{-1/2} \exp(-x^2/t)$ is uniformly continuous
(which one can check by Fourier transformation) and has the generator $\tilde B= B|_{\SSa(\R)}$. If $B$ would generate a semigroup $T$ on
$\CC^\infty(\R)$ we obtain, for fixed $f\in\SSa(\R)$ and $t>0$, that $\varphi(s)= T_s \circ \tilde T_{t-s}(f)$ has vanishing derivative
so that $T_t(f)=\varphi(t)=\varphi(0)=\tilde T_t(f)$. But this means that the convolution with the Gau{\ss} kernel can be continuously
extended from $\SSa(\R)$ to $\CC^\infty(\R)$ which is not true.

\bibliographystyle{amsalpha}
\bibliography{Semigroups}


\end{document}